\documentclass[a4paper,12pt,reqno]{amsart}

\usepackage[utf8]{inputenc}
\usepackage[top=25truemm,bottom=25truemm,left=20truemm,right=20truemm]{geometry}

\usepackage{amsmath}
\usepackage{amsthm,amsfonts}
\usepackage{amssymb}
\usepackage{stmaryrd}
\usepackage{mathtools}
\usepackage{mathrsfs} 
\usepackage[bbgreekl]{mathbbol} 
\usepackage{tikz-cd}
\usetikzlibrary{decorations.markings,decorations.pathmorphing,shapes.multipart,fit}
\usepackage{thmtools} 
\makeatletter
\@ifundefined{newcounteralias}{}{%
    \renewcommand\thmt@autorefsetup{%
        \@xa\def\csname\thmt@envname autorefname\@xa\endcsname\@xa{\thmt@thmname}%
  }%
}
\makeatother

\usepackage{enumitem}
\usepackage{CJKutf8}

\DeclareSymbolFontAlphabet{\mathbb}{AMSb}
\DeclareSymbolFontAlphabet{\mathbbl}{bbold}
\DeclareEmphSequence{\bfseries\itshape}

\renewcommand{\epsilon}{\varepsilon}
\renewcommand{\phi}{\varphi}

\setlist{nosep}
\setlist*[enumerate,1]{label*=(\roman*)}

\makeatletter
\newcommand{\labeleditem}[1]{
\item[\text{#1}]\protected@edef\@currentlabel{\text{#1}}\phantomsection
}
\makeatother

\usepackage[%
    backend=biber,%
    style=alphabetic,%
    sorting=nyt,%
    giveninits=true,
    backref=true,%
    isbn=false,%
    url=true,%
]{biblatex}
\makeatletter
\renewrobustcmd*{\mkbibemph}{\mkbibitalic}
\protected\long\def\blx@imc@mkbibemph#1{\blx@imc@mkbibitalic{#1}}
\makeatother

\definecolor{darkblue}{rgb}{0.2,0,0.6}
\definecolor{darkgreen}{rgb}{0.2,0.5,0.2}
\usepackage[%
    colorlinks=true,%
    linkcolor=darkblue,%
    citecolor=green,%
    urlcolor=darkgreen,%
    hypertexnames=false,%
    naturalnames=true,
    pdfusetitle=true,
]{hyperref}
\usepackage[%
    nameinlink,%
]{cleveref}

\crefname{equation}{}{}
\crefname{enumi}{}{}
\crefname{section}{Section}{Sections}
\crefname{appendix}{Appendix}{Appendices}
\crefname{figure}{Figure}{Figures}

\numberwithin{equation}{section}

\declaretheorem[%
    style=definition,%
    numberwithin=section,%
    qed=$\lrcorner$,%
]{definition}
\declaretheorem[%
    style=definition,%
    sibling=definition,%
    qed=$\lrcorner$,%
]{remark,example,notation,fact,construction}
\declaretheorem[%
    style=definition,%
    sibling=definition,%
    name={Open Problem}
]{openproblem}
\declaretheorem[%
    style=definition,%
    sibling=definition,%
]{theorem,lemma,corollary,proposition}
\declaretheorem[%
    style=definition,%
    numbered=no,%
    name=Theorem,%
]{theorem*}
\declaretheorem[%
    style=definition,%
    numbered=no,%
    name={Open Problem},%
]{openproblem*}
\declaretheorem[%
    style=definition,%
    numbered=no,%
    name={Question},%
]{question*}

\crefname{theorem}{Theorem}{Theorems}
\crefname{definition}{Definition}{Definitions}
\crefname{lemma}{Lemma}{Lemmas}
\crefname{corollary}{Corollary}{Corollaries}
\crefname{proposition}{Proposition}{Propositions}
\crefname{remark}{Remark}{Remarks}
\crefname{example}{Example}{Examples}
\crefname{notation}{Notation}{Notations}
\crefname{construction}{Construction}{Construction}
\crefname{claim}{Claim}{Claims}
\crefname{fact}{Fact}{Facts}

\usepackage{xpatch}
\makeatletter   
\xpatchcmd{\@tocline}
{\hfil\hbox to\@pnumwidth{\@tocpagenum{#7}}\par}
{\ifnum#1<2\hfill\else\dotfill\fi\hbox to\@pnumwidth{\@tocpagenum{#7}}\par}
{}{}
\def\l@subsection{\@tocline{2}{0pt}{2pc}{6pc}{}}
\makeatother    

\tikzcdset{equal/.append style={double distance=1.0pt,-}}

\tikzcdset{%
    arrow style=tikz,
    diagrams={>={Straight Barb[scale=0.8]}}
}

\tikzstyle{tri} = [row sep={1.6em}, column sep={0.8em}]
\tikzstyle{tinytri} = [row sep={0.8em}, column sep={0.1em}]

\tikzstyle{huge} = [row sep={3.6em}, column sep={3.6em}]
\tikzstyle{large} = [row sep={2.7em}, column sep={2.7em}]
\tikzstyle{normal} = [row sep={1.8em}, column sep={1.8em}]
\tikzstyle{scriptsize} = [row sep={1.35em}, column sep={1.35em}]
\tikzstyle{small} = [row sep={0.9em}, column sep={0.9em}]
\tikzstyle{tiny} = [row sep={0.45em}, column sep={0.45em}]

\tikzstyle{hugecolumn} = [column sep={3.6em}]
\tikzstyle{largecolumn} = [column sep={2.7em}]
\tikzstyle{normalcolumn} = [column sep={1.8em}]
\tikzstyle{scriptsizecolumn} = [column sep={1.35em}]
\tikzstyle{smallcolumn} = [column sep={0.9em}]
\tikzstyle{tinycolumn} = [column sep={0.45em}]

\tikzstyle{hugerow} = [row sep={3.6em}]
\tikzstyle{largerow} = [row sep={2.7em}]
\tikzstyle{normalrow} = [row sep={1.8em}]
\tikzstyle{scriptsizerow} = [row sep={1.35em}]
\tikzstyle{smallrow} = [row sep={0.9em}]
\tikzstyle{tinyrow} = [row sep={0.45em}]

\newcommand{\A}{\mathscr{A}}
\newcommand{\B}{\mathscr{B}}
\newcommand{\C}{\mathscr{C}}
\newcommand{\D}{\mathscr{D}}

\renewcommand{\I}{\mathscr{I}}

\renewcommand{\S}{\mathscr{S}}

\newcommand{\bN}{\mathbb{N}}

\newcommand{\Set}{\mathbf{Set}}

\newcommand{\Obj}{\mathrm{Obj}}
\newcommand{\Mor}{\mathrm{Mor}}

\newcommand{\id}{\mathsf{id}}

\newcommand{\Colim}{\mathop{\mathrm{Colim}}}
\newcommand{\El}{{\textstyle\int}}

\newcommand{\abs}[1]{{|#1|}}

\newcommand{\incat}[1]{\hspace{1em}\text{in }{#1}}

\renewcommand{\emptyset}{\varnothing}

\newcommand{\yoneda}{\text{\begin{CJK}{UTF8}{min}よ\end{CJK}}}

\NewDocumentCommand{\arr}{D(){} O{} O{1}}{\mathrel{%
    \begin{tikzcd}[column sep={#3em}, ampersand replacement=\&]
        \hspace{-2ex} \& \hspace{-2ex}
        \arrow[from=1-1, to=1-2, "{%
            \IfValueTF{#1}{%
                \raisebox{0pt}[3pt][\depth]{$\scriptstyle #1$}%
            }{%
                \scriptstyle #1%
            }%
        }", #2]
    \end{tikzcd}
}}
\NewDocumentCommand{\rra}{D(){} O{} O{1}}{\mathrel{%
    \begin{tikzcd}[column sep={#3em}, ampersand replacement=\&]
        \hspace{-2ex} \& \hspace{-2ex}
        \arrow[from=1-2, to=1-1, "{%
            \IfValueTF{#1}{%
                \raisebox{0pt}[3pt][\depth]{$\scriptstyle #1$}%
            }{%
                \scriptstyle #1%
            }%
        }"', #2]
    \end{tikzcd}
}}

\newcommand{\Ind}{\mathop{\mathbf{Ind}}}
\newcommand{\Sind}{\mathop{\mathbf{Sind}}}
\newcommand{\Rec}{\mathop{\mathbf{Rec}}}
\newcommand{\Psh}{\mathop\mathbf{Psh}}

\newcommand{\Pfin}{\mathop{\mathrm{\mathcal{P}_{fin}}}}
\newcommand{\Cospan}{\mathop{\mathbf{Cospan}}}
\newcommand{\card}[1]{\abs{#1}}

\tikzset{%
    header/.style={%
        label={%
            [rectangle, fill=white, draw, inner sep=3pt, anchor=north west]%
            north west:{#1}%
        }%
    },%
    frame/.style={%
        draw,%
        anchor=south west,%
        inner sep=0pt,%
    },%
}

\title{Cocompletion under sifted colimits need not be finitely accessible}

\author{Yuto Kawase}
\address{Research Institute for Mathematical Sciences, Kyoto University, Kyoto 606-8502, Japan}
\email{ykawase@kurims.kyoto-u.ac.jp}

\date{\today}
\keywords{%
    sifted colimit, %
    generalized variety, %
    finitely accessible category%
}
\thanks{%
The author is grateful to his supervisor, Masahito Hasegawa, for his support. This work was supported by JSPS KAKENHI Grant Number JP24KJ1462.%
}
\subjclass[2020]{%
    18C35
}

\begin{document}
\begin{abstract}
    We give a negative answer to the open problem posed by Ad\'{a}mek and Rosick\'{y} in 2001, asking whether the cocompletion of any small category under sifted colimits is finitely accessible.
    We will construct an infinite family of counterexamples.
\end{abstract}

\maketitle
\section{Introduction}
Locally finitely presentable categories provide a categorical framework for structures built from finitely presented pieces.
More precisely, a locally small category is \textit{locally finitely presentable} if it is cocomplete and has a small set of finitely presentable objects from which every object is obtained as a filtered colimit, where an object $C\in\C$ is \textit{finitely presentable} if the hom-functor $\C(C,-)$ preserves filtered colimits.
The latter condition on the existence of the small set is called \textit{finite accessibility} together with the existence of filtered colimits.

An algebraic counterpart is obtained by replacing filtered colimits with \textit{sifted colimits}.
Recall that filtered colimits are precisely colimits that commute with finite limits in $\Set$, while sifted colimits are colimits that commute with finite products in $\Set$.
A typical example of a sifted colimit is a \textit{reflexive coequalizer}, which is the coequalizer of a pair of parallel morphisms having a common section.
Replacing ``filtered'' with ``sifted'' leads to the corresponding notions of \textit{strongly finitely presentable objects} and \textit{finitary algebraic categories}.
Finitary algebraic categories are, up to equivalence, precisely many-sorted varieties, that is, the categories of models of a many-sorted finitary equational theory \cite{AdamekRosickyVitale2011algebraic}.

It is known that a category is locally finitely presentable precisely when it is finitely accessible and complete; see \cite[Corollary~2.47]{AdamekRosicky1994locally}.
Broader classes arise by requiring the existence of fewer limits.
For example, if completeness is replaced by the existence of connected limits, one obtains \textit{locally finitely multi-presentable categories} \cite{Diers1980multipresentable}; if no limits are required, one obtains finitely accessible categories.
The same relaxation on the sifted side leads from finitary algebraic categories, through \textit{multi-algebraic categories} \cite{Diers1980multialgebraic}, to \textit{generalized varieties} \cite{AdamekRosicky2001sifted}.
More generally, accessible categories with a class of limits are studied systematically in \cite{LackTendas2024accessiblelimit}.

Generalized varieties thus form the sifted analogue of finitely accessible categories.
It is known that every finitary algebraic category is locally finitely presentable \cite[Corollary~3.7]{AdamekRosicky1994locally}.
Additionally, every multi-algebraic category is locally finitely multi-presentable, since the former can be characterized by (finite product, coproduct)-sketches \cite[Lemma~1]{AdamekRosicky2001multivarieties} and the latter by (finite limit, coproduct)-sketches; see \cite[Theorem~4.32]{AdamekRosicky1994locally} and the remark below.
Whether this passage from the sifted side to the filtered side remains valid without any completeness assumption was first posed as an open problem in \cite[4.9(1)]{AdamekRosicky2001sifted}, and was later reiterated in \cite[Example~4.11(2)]{Tendas2025more}:
\begin{question*}[{\cite[4.9(1)]{AdamekRosicky2001sifted}}]
    Is every generalized variety finitely accessible?
\end{question*}\noindent

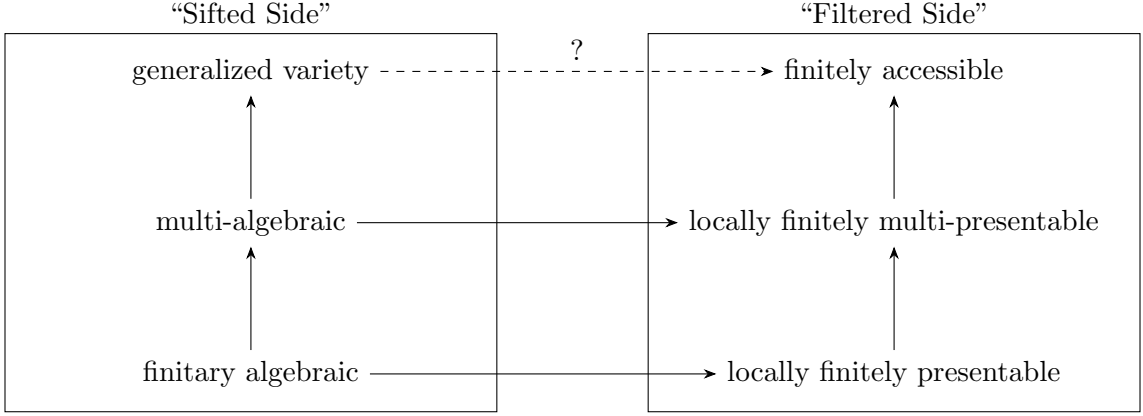
\begin{figure}[h]
    \centering
    \begin{tikzpicture}[>=Stealth, font=\small]
        \draw (0,0.5) rectangle (6.5,-4.5);
        \node[above] at (3.25,0.5) {``Sifted Side''};

        \node (gv) at (3.25,0) {generalized variety};
        \node (ma) at (3.25,-2) {multi-algebraic};
        \node (falg) at (3.25,-4) {finitary algebraic};

        \draw[->] (ma) -- (gv);
        \draw[->] (falg) -- (ma);

        \draw (8.5,0.5) rectangle (15,-4.5);
        \node[above] at (11.75,0.5) {``Filtered Side''};

        \node (facc) at (11.75,0) {finitely accessible};
        \node (lfmp) at (11.75,-2) {locally finitely multi-presentable};
        \node (lfp)  at (11.75,-4) {locally finitely presentable};
        
        \draw[->] (lfmp) -- (facc);
        \draw[->] (lfp) -- (lfmp);

        \draw[->,dashed] (gv.east) -- node[above] {?} (facc.west);
        \draw[->] (ma.east) -- (lfmp.west);
        \draw[->] (falg.east) -- (lfp.west);
    \end{tikzpicture}
    \caption{Implications between several classes of accessible categories}
    \label{fig:implications-sifted-filtered}
\end{figure}

It is already known that every generalized variety is $\aleph_1$-accessible; see \cite[Remark~4.8(2)]{AdamekRosicky2001sifted} and also \cite[Example~4.11(2)]{Tendas2025more}.
Thus, the remaining question is whether the accessibility rank can always drop from $\aleph_1$ to $\aleph_0$.

To approach this question, it is useful to work with cocompletions.
Recall that for a small category $\A$, its cocompletion under any given class of colimits can be constructed as the full subcategory of the presheaf category $\Psh(\A)$ obtained by closing the representables under colimits in that class iteratively.
Let $\Sind(\A),\Ind(\A)\subseteq\Psh(\A)$ be the cocompletions under sifted and filtered colimits, respectively.
Explicitly, a presheaf over $\A$ belongs to $\Sind(\A)$ precisely when the category of elements is sifted \cite[Theorem~2.6]{AdamekRosicky2001sifted}, just as it belongs to $\Ind(\A)$ precisely when the category of elements is filtered.

Then, a category is a generalized variety precisely when it is equivalent to $\Sind(\A)$ for some small category $\A$ \cite[Remark~3.9(1)]{AdamekRosicky2001sifted}, just as a category is finitely accessible precisely when it is equivalent to $\Ind(\A)$.
This formulation enables us to rephrase the question as follows:
\begin{question*}
    Is $\Sind(\A)$ finitely accessible for every small category $\A$?
\end{question*}

Several additional assumptions ensure an affirmative answer.
If $\A$ is finite, $\Sind(\A)$ becomes essentially finite.
Since an essentially finite category is finitely accessible if and only if it is Cauchy complete, $\Sind(\A)$ is finitely accessible, hence the answer to the question is positive in this case.
We will discuss this finite case in \cref{sec:further}.

If a small category $\A$ has finite coproducts or pullbacks, then
\begin{equation*}
    \Sind(\A) \simeq \Ind(\Rec(\A)),
\end{equation*}
where $\Rec(\A)$ is the cocompletion of $\A$ under reflexive coequalizers.
This decomposition also leads to the positive answer, showing that $\Sind(\A)$ is finitely accessible.
The finite-coproduct case first appeared in \cite[Example~2.3(2)]{AdamekRosicky2001sifted}; for a complete proof, see \cite[Corollary~7.4]{AdamekRosickyVitale2011algebraic}.
The pullback case is due to Chen \cite[Theorem~5.1]{Chen2021sifted}.

The above equivalence can fail for a small category in general, as first shown in \cite[Example~2.3(4)]{AdamekRosicky2001sifted}.
However, this does not exclude a presentation as $\Ind(\B)$ for some other small category $\B$.
Indeed, the finite accessibility of $\Sind(\D)$, where $\D$ is the category as in \cref{eg:twin-reflexive-pairs}, does not stem from the Ind--Rec decomposition.

In this paper, we answer the open problem negatively:
\begin{theorem*}
    There exists a small category $\A$ such that $\Sind(\A)$ is not finitely accessible.
\end{theorem*}\noindent
The category $\A$ witnessing the theorem will be constructed in \cref{def:construction-of-A}.
In fact, our construction yields not just a single counterexample, but an infinite family of them.
The construction starts from an arbitrary small sifted category $\S$ that is not filtered, for example, the walking reflexive pair.
We carry out the construction and prove the theorem in \cref{sec:counterexample}.

Our strategy is to construct $\A$ so that the only finitely presentable objects of $\Sind(\A)$ are the representables, while some objects lie outside their closure under filtered colimits.
The central idea is to express the terminal presheaf as a chain colimit in $\Sind(\A)$ in two ``orthogonal'' ways.
Finite presentability forces the unique map from any finitely presentable object to the terminal object to factor through some stage of each chain.
We arrange the two chains so that these simultaneous factorizations impose a strong restriction, which is illustrated in \cref{fig:two-orthogonal-chains}.
The key ingredient is \cref{stat:span-implies-represn}.

\section{Counterexamples}\label{sec:counterexample}
We fix a small sifted but non-filtered category $\S$; for example, the walking reflexive pair, which is given by
\begin{equation*}
    \begin{tikzcd}
        1\ar[r,shift left=3,"d_0"]\ar[r,shift right=3,"d_1"']
            &
            0\ar[l,"s"{description}]
    \end{tikzcd}
    \quad\text{with}\quad
    d_i\circ s = \id_0
    \quad (i=0,1).
\end{equation*}
Note that the following definition depends on $\S$.
\begin{construction}[The category $\A$]\label{def:construction-of-A}
    Let $\A$ denote the category defined as follows.
    \begin{description}[labelindent=1em]
        \item[Object]
            $\Obj\A\coloneq \{(X,n)\mid \text{$X\in\S$ and $n\ge 0$}\} = (\Obj\S)\times\bN$.
        \item[Morphism]
            \begin{gather*}
                \A((X,m),(Y,n))\coloneq
                \begin{cases}
                    \left\{
                        (s,F)
                        \,\middle|\,
                        \begin{aligned}
                            &\text{$X\arr(s)Y$ in $\S$, $F\in\Pfin(\bN)\setminus\{\emptyset\}$}
                            \\
                            &\text{such that $n\le\max F$}
                        \end{aligned}
                    \right\}
                        &
                        \text{if $m<n$;}
                    \\
                    \left\{
                        (\id_X,\emptyset)
                    \right\}
                        &
                        \text{if $X=Y$ and $m=n$;}
                    \\
                    \emptyset
                        &
                        \text{otherwise.}
                \end{cases}
            \end{gather*}
            Here, $\Pfin(\bN)$ denotes the set of finite subsets of $\bN$, and $\max F$ denotes the largest element of $F$.
            Define $\id_{(X,m)}\coloneq (\id_X,\emptyset)$.
        \item[Composition]
            The composite of
            \begin{equation*}
                (X_0,n_0)\arr((s_1,F_1))[][2] (X_1,n_1)\arr((s_2,F_2))[][2] (X_2,n_2)
                \incat{\A}
            \end{equation*}
            is defined to be $(s_2\circ s_1, F_1\cup F_2)$.
            This actually defines a morphism $(X_0,n_0)\arr() (X_2,n_2)$.
            Indeed, if the second morphism is non-identity, then $n_2\le\max F_2\le\max(F_1\cup F_2)$; otherwise, there is nothing to check.
            Associativity and the unit laws follow immediately from those in $\S$ and from those of union, with $\emptyset$ as its unit.\qedhere
    \end{description}
\end{construction}

\begin{notation}[The subcategories $\A_S$]
    For a subset $S\subseteq\bN$, let $\A_S$ denote the subcategory of $\A$ with the same objects as $\A$ where a morphism $(X,m)\arr((s,F))[][2](Y,n)$ in $\A$ belongs to $\A_S$ if $F\subseteq S$.
\end{notation}

\begin{notation}[The presheaves $P_S$]\label{def:P_S}
    For a subset $S\subseteq\bN$, let $P_S$ denote the presheaf obtained by taking the colimit of the diagram $\A_S\arr \A\arr(\yoneda)[hook] \Psh(\A)$, the composite of the inclusion and the Yoneda embedding.
    Explicitly, for each object $A\in\A$, an element of $P_S(A)$ is an equivalence class $[f]_S$ of a morphism $A\arr(f)\cdot$ in $\A$ under the smallest equivalence relation $\sim_S$ satisfying the following:
    \begin{center}
        If $k\circ f=g$ for some $k$ belonging to $\A_S$, then $f\sim_S g$.
    \end{center}
    Note that the presheaf $P_\bN$ becomes a terminal presheaf.
\end{notation}

\begin{lemma}\label{stat:isFiltered-A_S}
    For every infinite $S\subseteq\bN$, the category $\A_S$ is sifted.
\end{lemma}
\begin{proof}
    Non-emptiness is trivial.
    Let $(X_0,n_0)$ and $(X_1,n_1)$ be objects in $\A_S$.
    We have to show that these extend to a cospan.
    Since $\S$ is sifted, there exists a cospan $X_0\arr(s_0) Y\rra(s_1) X_1$ in $\S$.
    Since $S$ is infinite, we can take a sufficiently large $N\in S$ and obtain a cospan $(X_i,n_i)\arr((s_i,\{N\}))[][3] (Y,N)$ in $\A_S$.

    Suppose that we are given two cospans $(X_i,n_i)\arr((s_i,F_i))[][3] (Y,p)$ and $(X_i,n_i)\arr((t_i,G_i))[][3] (Z,q)$ in $\A_S$, where $i=0,1$.
    Since $\S$ is sifted, there are cospans $X_i\arr(u_{k,i}) V_k$, and a zig-zag $W_{k-1}\rra(a_k) V_k \arr(b_k) W_k$ $(1\le k \le l)$ such that
    \begin{gather*}
        W_0 = Y, \quad W_l = Z; \\
        a_1\circ u_{1,i} = s_i \quad (i=0,1); \\
        b_l\circ u_{l,i} = t_i \quad (i=0,1); \\
        b_k\circ u_{k,i} = a_{k+1}\circ u_{k+1,i} \quad (1\le k < l;~i=0,1)
        \incat{\S}.
    \end{gather*}
    Since $S$ is infinite, we can take some $M,N\in S$ such that $p,q<M<N$.
    Let $L\coloneq F_0\cup F_1 \cup G_0 \cup G_1 \cup \{M,N\}$.
    Then, the following diagram commutes in $\A_S$ for each $i=0,1$:
    \begin{equation*}
        \begin{tikzcd}[scriptsize]
            (W_0,N)
                &
                (V_1,M)\ar[l,"{(a_1,L)}"']\ar[r,"{(b_1,L)}"]
                    &
                    (W_1,N)
                        &
                        (V_2,M)\ar[l,"{(a_2,L)}"']\ar[r,"{(b_2,L)}"]
                            &
                            \cdots
                                &
                                (V_l,M)\ar[l,"{(a_l,L)}"']\ar[r,"{(b_l,L)}"]
                                    &
                                    (W_l,N)
            \\
            (Y,p)\ar[u,"{(\id,L)}"]
                &
                    &
                        &
                            &
                                &
                                    &
                                    (Z,q)\ar[u,"{(\id,L)}"']
            \\[2em]
                &
                    &
                        &
                        (X_i,n_i)
                        \ar[lllu,"{(s_i,F_i)}"]
                        \ar[lluu,"{(u_{1,i},L)}"]
                        \ar[uu,"{(u_{2,i},L)}"]
                        \ar[rruu,"{(u_{l,i},L)}"']
                        \ar[rrru,"{(t_i,G_i)}"']
                            &
                                &
                                    &
        \end{tikzcd}
    \end{equation*}
    This shows that the cospans $(s_i,F_i)_{i=0,1}$ and $(t_i,G_i)_{i=0,1}$ are connected to each other; hence $\A_S$ is sifted.
\end{proof}

\begin{proposition}\label{stat:inSind-notinInd-P_S}
    For every infinite $S\subseteq\bN$, $P_S\in\Sind(\A)\setminus\Ind(\A)$.
\end{proposition}
\begin{proof}
    By \cref{stat:isFiltered-A_S}, $P_S$ is a sifted colimit of representables, hence $P_S\in\Sind(\A)$.
    The rest we have to show is $P_S\notin\Ind(\A)$, equivalently, the category of elements $\El P_S$ is not filtered.
    Since $\S$ is sifted and non-filtered, there exists a pair $(s_0,s_1)\colon X\arr Y$ of parallel morphisms in $\S$ that is never coequalized.
    Since $S$ is infinite, we can take $1\le N\in S$.
    Consider objects $A\coloneq (X,0)$, $B\coloneq (Y,1)$, and morphisms $f_i\coloneq (s_i,\{N\}) \colon A\arr B$ in $\A$.
    Since both $f_0$ and $f_1$ lie in $\A_S$, we have $[f_0]_S = [\id_A]_S = [f_1]_S$, where each $[\cdot]_S$ are equivalence classes as in \cref{def:P_S}.
    Thus, we obtain a pair of parallel morphisms $(A,[\id_A]_S)\arr(f_i) (B,[\id_B]_S)$ in $\El P_S$.
    Moreover, this pair still cannot be coequalized in $\El P_S$ because of the projection functor $\El P_S \arr \A \arr \S$.
    This shows that $\El P_S$ is not filtered.
\end{proof}

\begin{lemma}\label{stat:max_is_teminal_in_sifted}
    Let $\C$ be a sifted category, and let $Z\in\C$ be an object such that every morphism from $Z$ is the identity.
    Then, $Z$ is the terminal object in $\C$.
\end{lemma}
\begin{proof}
    Let $X\in\C$ be an arbitrary object.
    By the assumption on $Z$, every cospan from $X$ and $Z$ is necessarily of the form $(f,\id)$.
    Moreover, every morphism between two such cospans is the identity on $Z$.
    Thus, the category $\Cospan(X,Z)$ of cospans from $X$ and $Z$ is isomorphic to the discrete category of morphisms $X\arr Z$.
    However, this category must be connected since $\C$ is sifted.
    Therefore, there exists a unique morphism $X\to Z$.
\end{proof}

The following lemma plays a crucial role in providing a counterexample.
\begin{lemma}\label{stat:span-implies-represn}
    Let $S,T \subseteq\bN$ be infinite subsets such that the intersection $S\cap T$ is finite.
    Suppose that we are given $P\in\Sind(\A)$ and a span $P_S\rra P\arr P_T$ in $\Sind(\A)$.
    Then, the presheaf $P$ is representable.
\end{lemma}
\begin{proof}
    We first define a presheaf $Q_U\in\Psh(\A)$ for an arbitrary subset $U\subseteq\bN$ as follows.
    For every object $A\in\A$, let $Q_U (A) \coloneq \Pfin(\bN\setminus U)$, the set of all finite subsets of $\bN\setminus U$.
    For a morphism $f=(s,F)$ in $\A$, define a map $Q_U(f)\colon\Pfin(\bN\setminus U)\arr \Pfin(\bN\setminus U)$ by $K\arr[mapsto] (F\setminus U)\cup K$.

    We next construct a natural transformation $P_U\arr Q_U$ as follows.
    For each object $A\in\A_U$, the element $\emptyset\in Q_U(A)$ determines a natural transformation $\yoneda A\arr Q_U$ by the Yoneda lemma.
    These are compatible with the morphisms in $\A_U$.
    Indeed, if $f=(s,F)$ is a morphism in $\A_U$, then $F\subseteq U$, and hence $Q_U(f)(\emptyset) = (F\setminus U) \cup \emptyset = \emptyset$.
    Thus, by the universal property of the colimit, we obtain $P_U\arr Q_U$ in $\Psh(\A)$.

    Composing the given morphisms $P\arr P_S$ and $P\arr P_T$ with the natural transformations constructed above, we obtain a span $Q_S \rra(\alpha) P \arr(\beta) Q_T$ in $\Psh(\A)$.
    Since the category of elements $\El P$ is sifted, it is non-empty; hence there exists at least one element $x\in P(A)$ with $A\in\A$.
    Let
    \begin{equation*}
        H\coloneq (S\cap T) \cup \alpha_A(x) \cup \beta_A (x) \quad\subseteq\bN.
    \end{equation*}
    Let $(A,x)\arr(f) (B,y)$ be an arbitrary morphism from $(A,x)$ in $\El P$.
    Express $A=(X,m)$, $B=(Y,n)$, and $f=(s,F)$, respectively.
    By the naturality of $\alpha$ at $f$ as shown below, we have $F\setminus S \subseteq \alpha_A(x)$.
    Similarly, the naturality of $\beta$ at $f$ implies $F\setminus T \subseteq \beta_A(x)$.
    \begin{equation*}
        \begin{tikzcd}
            P(A)\ar[ddd,"\alpha_A"']
                &[-4em]
                    &
                        &
                            &[-4em]
                            P(B)\ar[llll,"P(f)"']\ar[ddd,"\alpha_B"]
            \\[-1.5em]
                &
                x\ar[d,mapsto]
                    &
                        &
                        y\ar[ll,mapsto]\ar[d,mapsto]
                            &
            \\
                &
                \alpha_A(x)
                    &
                    (F\setminus S)\cup\alpha_B(y)\ar[l,phantom,"="]
                        &
                        \alpha_B(y)\ar[l,mapsto]
                            &
            \\[-1.5em]
            Q_S(A)
                &
                    &
                        &
                            &
                            Q_S(B)\ar[llll,"Q_S(f)"]
        \end{tikzcd}
    \end{equation*}
    Then, we have $F\subseteq H$.
    The set $H$ is finite, since $S\cap T$, $\alpha_A(x)$, and $\beta_A(x)$ are all finite; moreover it is independent of the choice of $f$.
    If $f$ is non-identity, then $n\le \max F\le \max H$; hence, there are only finitely many possible values of $n$.
    Therefore, among all morphisms from $(A,x)$ in $\El P$, we can choose one
    \[
        (A,x)\arr() (B_\mathrm{max},y_\mathrm{max}),
        \qquad B_\mathrm{max}=(Y_\mathrm{max},n_\mathrm{max}),
    \]
    such that $n_\mathrm{max}$ is as large as possible.
    By the choice of $n_\mathrm{max}$, every morphism from $(B_\mathrm{max},y_\mathrm{max})$ in $\El P$ is the identity.
    Then, \cref{stat:max_is_teminal_in_sifted} implies that $\El P$ has a terminal object, hence $P$ is representable.
\end{proof}

\begin{proposition}\label{stat:fp_is_represn}
    Every finitely presentable object in $\Sind(\A)$ is representable.
\end{proposition}
\begin{proof}
    Let $E\subseteq\bN$ be the subset of even numbers and $O\subseteq\bN$ be the subset of odd ones.
    For each $n\in\bN$, let
    \begin{gather*}
        U_n \coloneq E \cup \{2k+1\mid 0\le k< n\} ~\subseteq\bN,\\
        V_n \coloneq \{2k \mid 0\le k< n\} \cup O ~\subseteq\bN.
    \end{gather*}
    Then, we obtain two chains $P_{U_0}\arr P_{U_1}\arr \cdots$ and $P_{V_0}\arr P_{V_1}\arr \cdots$ in $\Psh(\A)$; moreover, these chains lie in $\Sind(\A)$ by \cref{stat:inSind-notinInd-P_S}.
    It follows that every morphism $f=(s,F)$ in $\A$ eventually belongs to some $\A_{U_m}$ and $\A_{V_n}$, since $\bigcup_{n\in\bN} U_n = \bigcup_{n\in\bN} V_n = \bN$ and $F$ is finite by the construction of $\A$.
    Thus, we have $\Colim_{n<\omega} P_{U_n} \cong \Colim_{n<\omega} P_{V_n} \cong P_\bN$ in $\Psh(\A)$.
    These isomorphisms still hold in $\Sind(\A)$ because the full subcategory $\Sind(\A)\subseteq\Psh(\A)$ is, in general, closed under sifted colimits; in particular, chain colimits.

    Let $P\in\Sind(\A)$ be an arbitrary finitely presentable object.
    Then, since $\Colim_{n<\omega} P_{U_n} \cong \Colim_{n<\omega} P_{V_n} \cong P_\bN$ in $\Sind(\A)$, the unique morphism $P\arr P_\bN$ to the terminal object factors through $P_{U_m}$ and $P_{V_n}$ for some $m,n\in\bN$.
    Since $U_m\cap V_n$ is finite, $P$ is representable by \cref{stat:span-implies-represn}.
\end{proof}

\begin{theorem}
    The category $\Sind(\A)$ is not finitely accessible.
\end{theorem}
\begin{proof}
    For any infinite subset $S\subseteq\bN$, we already have $P_S\in\Sind(\A)\setminus\Ind(\A)$ by \cref{stat:inSind-notinInd-P_S}.
    Then, any of $P_S$ cannot be written as a filtered colimit of finitely presentable objects in $\Sind(\A)$ by \cref{stat:fp_is_represn}, which concludes that $\Sind(\A)$ is not finitely accessible.
\end{proof}

\begin{figure}[h]
    \centering
    \begin{tikzpicture}[>={Straight Barb[scale=0.8]},x=2em,y=1.6em]
        \coordinate (SindNE) at (12.5,9.8);
        \coordinate (SindSW) at (0,-2);

        \coordinate (IndNE) at (8.5,7);
        \coordinate (IndSW) at (0.75,-1);

        \coordinate (ANE) at (6.5,5);
        \coordinate (ASW) at (1.5,0);

        \node[
            frame,
            fit=(SindSW)(SindNE),
            header={$\Sind(\A)$}
        ] {};

        \node[
            frame,
            fit=(IndSW)(IndNE),
            header={$\Ind(\A)$}
        ] {};

        \node[
            frame,
            fit=(ASW)(ANE),
            header={$\A$}
        ] {};

        \node (PU0) at (1,8) {$P_{U_0}$};
        \node (PU1) at (3,8) {$P_{U_1}$};
        \node (Hdots1) at (5,8) {$\cdots$};
        \node (PUm) at (7,8) {$P_{U_m}$};
        \node (Hdots2) at (9,8) {$\cdots$};
        \node (PN) at (10,8) {$P_{\bN}$};
        \node (Cong) at (10.75,8) {$\cong$};
        \node (T) at (11.5,8) {$1$};

        \draw[->] (PU0) -- (PU1);
        \draw[->] (PU1) -- (Hdots1);
        \draw[->] (Hdots1) -- (PUm);
        \draw[->] (PUm) -- (Hdots2);

        \node (Vdots2) at (10,7) {$\vdots$};
        \node (PVn) at (10,5) {$P_{V_n}$};
        \node (Vdots1) at (10,3) {$\vdots$};
        \node (PV1) at (10,1) {$P_{V_1}$};
        \node (PV0) at (10,-1) {$P_{V_0}$};

        \draw[->] (PVn) -- (Vdots2);
        \draw[->] (Vdots1) -- (PVn);
        \draw[->] (PV1) -- (Vdots1);
        \draw[->] (PV0) -- (PV1);

        \node (P) at (5,3) {$P$};

        \draw[->] (P) -- (PUm);
        \draw[->] (P) -- (PVn);
    \end{tikzpicture}
    \caption{The two ``orthogonal'' chains in $\Sind(\A)$}
    \label{fig:two-orthogonal-chains}
\end{figure}
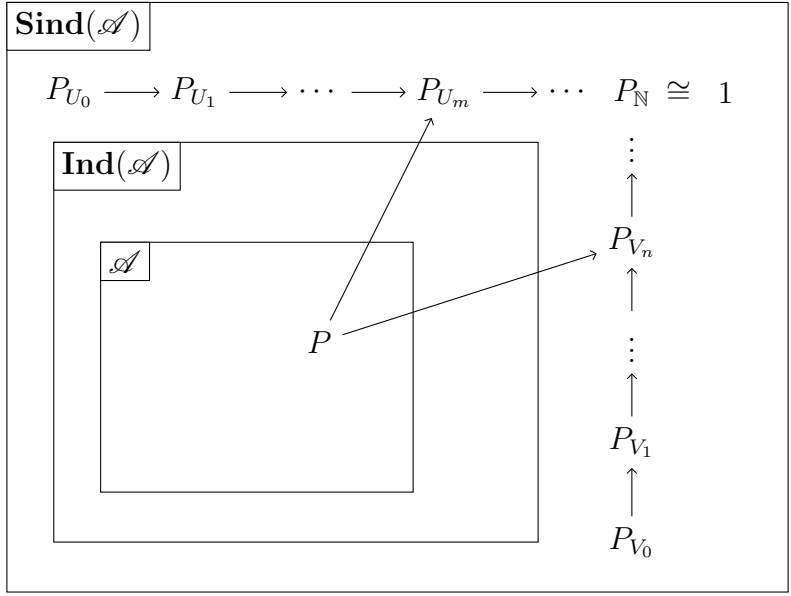

\section{The cocompletion of finite categories}\label{sec:further}
In this section, we remark that the cocompletion of any finite category under sifted colimits is always finitely accessible.
\begin{definition}
    We say a category \emph{admits cospans} if for any objects $I$ and $I'$, there exists a cospan $I\arr \cdot \rra I'$.
\end{definition}

We write $\card{X}$ for the cardinality of a set $X$.
\begin{lemma}\label{stat:less-than-N-colimit}
    Let $\I\arr(D)\Set$ be a functor from a small category admitting cospans, and let $X\in\Set$ be its colimit.
    Suppose there exists $N>0$ such that for every $I\in\I$, $\card{D(I)} < N$ holds.
    Then, $\card{X} < N$ holds.
\end{lemma}
\begin{proof}
    Take elements $[I_1,x_1],\dots,[I_N,x_N]$ in $X$ represented by $x_1\in D(I_1),\dots,x_N\in D(I_N)$.
    Since $\I$ admits cospans, we have an object $J\in\I$ and morphisms $I_k\arr(f_k)J$ $(1\le k\le N)$.
    By $\card{D(J)} <  N$, the elements $D(f_k)(x_k)$ $(1\le k\le N)$ are not distinct, hence so are $[I_k,x_k]=[J, D(f_k)(x_k)]$ $(1\le k\le N)$.
\end{proof}

The following proposition seems standard, but we include a complete proof for convenience.
The proof is based on an argument given on MathOverflow \cite{Simon2026filtered}.
\begin{proposition}\label{stat:finite-accessible}
    For a finite category $\C$, the following conditions are equivalent:
    \begin{enumerate}
        \item\label{stat:finite-accessible/ind}
            The ind-completion $\C\arr[hook]\Ind(\C)$ is an equivalence. 
        \item\label{stat:finite-accessible/finacc}
            $\C$ is finitely accessible.
        \item\label{stat:finite-accessible/acc}
            $\C$ is accessible.
        \item\label{stat:finite-accessible/filtered}
            $\C$ has filtered colimits.
        \item\label{stat:finite-accessible/cauchy}
            $\C$ is Cauchy complete.
    \end{enumerate}
\end{proposition}
\begin{proof}
    The directions \cref*{stat:finite-accessible/ind} $\implies$ \cref*{stat:finite-accessible/finacc} $\implies$ \cref*{stat:finite-accessible/acc} $\implies$ \cref*{stat:finite-accessible/cauchy} and \cref*{stat:finite-accessible/finacc} $\implies$ \cref*{stat:finite-accessible/filtered} $\implies$ \cref*{stat:finite-accessible/cauchy} are trivial.
    We show \cref*{stat:finite-accessible/cauchy} $\implies$ \cref*{stat:finite-accessible/ind} below.
    Consider the Yoneda embedding $\C\arr(\yoneda)[hook]\Psh(\C)$.
    Let $D\colon\I\to \C$ be a functor from a small filtered category, and let $X\in\Psh(\C)$ be a colimit of $\I\arr(D)\C\arr(\yoneda)[hook]\Psh(\C)$.
    Our goal is to show $X$ is representable.

    We can show that $X$ is finitely presentable in $\Psh(\C)$ as follows.
    Since any filtered category admits cospans, \cref{stat:less-than-N-colimit} shows that $X(C)$ is finite for every $C\in\C$.
    Combining this with the finiteness of $\C$, we conclude that $X$ is finitely presentable.
    
    Thus, the identity $X\arr(\id) \Colim_{I\in\I}\yoneda D(I)$ factors through some coprojection.
    \begin{equation*}
        \begin{tikzcd}[small]
            X\ar[rr,equal]\ar[rd,dotted] && X \\
            & \yoneda D(I)\ar[ru,"\mathrm{coproj}_I"'] &
        \end{tikzcd}\incat{\Psh(\C)}
    \end{equation*}
    Thus, $X$ is a retract of a representable one.
    Since $\C$ is Cauchy complete, $X$ is representable.
\end{proof}

\begin{corollary}\label{stat:finite-sind-is-finacc}
    For every finite category $\C$, $\Sind(\C)$ is essentially finite and finitely accessible.
\end{corollary}
\begin{proof}
    Let $N\coloneq \card{\Mor\C} + 1$.
    Take $X\in\Sind(\C)$ arbitrarily.
    Since $X$ is a sifted colimit of representables and any sifted category admits cospans, \cref{stat:less-than-N-colimit} shows $\card{X(C)} < N$.
    Therefore, $\Sind(\C)$ has only finitely many objects up to isomorphism, and for each pair of objects, there are only finitely many morphisms between them.
    Since $\Sind(\C)$ is always Cauchy complete, \cref{stat:finite-accessible} implies that $\Sind(\C)$ is finitely accessible.
\end{proof}

\begin{example}\label{eg:twin-reflexive-pairs}
    Let $\D$ denote the category obtained by
    \begin{equation*}
        \begin{tikzcd}
            1\ar[rd,shift left=3,"d_0"]\ar[rd,shift right=3,"d_1"']
                &
                    &
                    1'\ar[ld,shift left=3,"d'_1"]\ar[ld,shift right=3,"d'_0"']
            \\
                &
                0\ar[ul,"s"{description}]\ar[ur,"s'"{description}]
        \end{tikzcd}
        \quad\text{with}\quad
        d_i\circ s = \id_0 = d'_i\circ s'
        \quad (i=0,1).
    \end{equation*}
    Since $\D$ is finite, \cref{stat:finite-sind-is-finacc} implies that $\Sind(\D)$ is finitely accessible.

    The finite accessibility of $\Sind(\D)$ does not stem from the Ind--Rec decomposition.
    Indeed, $\Rec(\D)$ can be obtained by adding coequalizers of both $(d_0,d_1)$ and $(d'_0,d'_1)$, and $\Ind(\Rec(\D)) \simeq \Rec(\D)$ by combining \cref{stat:finite-accessible} and the fact that every idempotent splits in $\Rec(\D)$.
    However, $\Sind(\D)$ additionally contains the joint coequalizer of $(d_0,d_1)$ and $(d'_0,d'_1)$ since $\D$ itself is sifted as shown in \cite[Example~1.3]{AdamekRosickyVitale2010what}.
    Therefore, the equivalence $\Sind(\D)\simeq\Ind(\Rec(\D))$ fails for this example.
\end{example}

The preceding results show that the (essentially) finiteness of $\C$ is sufficient for $\Sind(\C)$ to be finitely accessible.
It is natural to ask whether this assumption can be weakened from finiteness to finite presentability.
We end the paper with the following question.
\begin{openproblem}
    If $\C$ is a finitely presented category, is $\Sind(\C)$ finitely accessible?
\end{openproblem}

\section*{Use of generative AI in the research process}
During the development of the counterexample presented in \cref{sec:counterexample}, the author used ChatGPT (GPT-6.1 Sol, OpenAI) to explore and refine candidate constructions following a strategy devised by the author. The resulting construction and all mathematical arguments were independently verified and revised by the author. The author takes full responsibility for the mathematical content of this article.

\printbibliography

\end{document}